\documentclass[12pt]{amsart}
\usepackage[utf8]{inputenc}
\usepackage{amsmath,amsthm,amscd,amssymb,bbm,graphicx}
\usepackage{floatflt,setspace, wrapfig}
\usepackage{graphics,multicol}
\usepackage{mathrsfs} 
\usepackage{epstopdf}
\usepackage[letterpaper,
            bindingoffset=0.2in,
            left=1in,
            right=1in,
            top=1in,
            bottom=1in,
            footskip=0.5in]{geometry}
\usepackage{todonotes}
\usepackage{bbm}
\usepackage{xcolor}

\newtheorem{theorem}{Theorem}[section]

\newtheorem*{corollary}{Corollary}

\newtheorem{lemma}{Lemma}[section]

\DeclareMathOperator{\diam}{diam}

\providecommand{\abs}[1]{\lvert \, #1 \, \rvert}

\newcommand{\bea}[1]{\begin{eqnarray}\label{#1}}
\newcommand{\eea}{\end{eqnarray}}

\allowdisplaybreaks[4]

\usepackage{subfiles}

\title{Higher order return times for $\phi$-mixing measures}

\begin{document}
 \maketitle
\authors{Nicolai T A Haydn\footnote{Department of Mathematics, University of Southern California,
Los Angeles, 90089-2532. E-mail: {\tt \email{nhaydn@usc.edu}}.},
Gin Park\footnote{E-mail: {\tt \email{gin\_park@icloud.com}}.}}


\maketitle

\pagenumbering{roman}

\section*{Abstract}
In this paper we consider $\phi$-mixing measures and show that the limiting return times
distribution is compound Poisson distribution as the target sets shrink to a zero measure
set. The approach we use generalises a method given by Galves and Schmitt in 1997 
for the first entry time to higher orders.


\pagenumbering{arabic}

\section{Introduction}

Entry times distributions have been studied mainly for the last three decades primarily
starting with a paper by Pitskel~\cite{Pit91} for Axiom A maps using symbolic dynamics
for equilibrium states.
A similar result was shortly after that also obtained by Hirata~\cite{Hir93} using the 
Laplace transform and a spectral approach to the transfer operator. He then
concluded from this Poisson distributed higher order return times invoking the
weak mixing property. Pitskel also found that at periodic points the limiting
distribution of first return time is not a straight exponential distribution but 
has a point mass at $0$ which corresponds to the periodicity and whose 
value is given by the Birkhoff sum of the normalised potential function over the 
period. For interval maps Collet~\cite{Col96} obtained similar results.

In 1997 Galves and Schmitt~\cite{GS97} gave a more general result for exponentially
$\psi$-mixing measures in a symbolic setting  where the return times distributions 
are determined in cylinder neightbourhoods of points. For generic i.e.\ non-periodic
points this then gives in the limit an exponential distribution and at periodic points 
one recovers a Pitskel type result. 
This method was then greatly extended by Abadi~\cite{Ab06,AS11} to $\alpha$-mixing measures.
 By other methods Abadi and Vergnes~\cite{Ab08,AV08}) then also proved the limiting 
 return statistics at generic points to be Poissonian for $\phi$-mixing measures.
 Using a more direct hands-on approach Poisson distribution was also proven in~\cite{HSV99}
 for $\phi$-mixing measures.
 Denker, Gordin and Sharova~\cite{DGS04} used the Chen-Stein method to find the limiting
 return times distribution of toral maps. The special setting allowed to apply harmonic analysis
 and found a dichotomy between generic, non-periodic, points  and periodic points. 
In~\cite{HV09} we then found that for $\psi$-mixing systems one gets a Poisson
distribution for generic points and in the case of periodic points one obtains a 
P\'olya-Aeppli distribution. We used a moment method which in the periodic case 
was suitably adapted and used Cauchy estimates to get  bounds on the error terms.

In~\cite{HV20,GHV24} we expanded the setting to allow the limiting set to be any null-set
and not necessarily a single point where we cover two different settings: geometric ball-like
neighbourhoods where we assume that the invariant measure has polynomial decay 
of correlations and $\phi$-mixing measures where we use the Chen-Stein method 
for compound Poisson distribution.

In Section~\ref{cluster.probabilities} we describe the probabilities that determine the 
likelihood of clusters of a given length to occur. The main result of that section is
Theorem~\ref{sequence.theorem} which allows the parameters to be computed by
a single limit rather than by executing a double limit. For this theorem we require the 
mixing to be $\phi$-mixing rather than $\alpha$-mixing in which case the decay rate 
necessary would have to be very rapid. In Section~\ref{convolution} we then give the 
convolution formula that allows to do a blocking argument for higher order returns.
The special case of no entries is the formula originally used by Galves and Schmitt 
and subsequently by others. In Section~\ref{generating.function} we use the blocking 
argument to related higher order return probabilities to those on short blocks. 
In the final section we comment in particular on the standard Kac scaling (with 
time adjustment) in case when the coefficient $\alpha_1$ (extremal index) is non zero.


\section{Return times and cluster probabilities}\label{cluster.probabilities}

Let $\Omega$ be a measurable space with $T:\Omega\to\Omega$ and $\mu$ a $T$-invariant
probability measure. 
For $U\subset \Omega$ so that $\mu(U)>0$ define the counting function $Z^N_U:\Omega\to \mathbb{N}_0$ by
$$
Z^N_U= \sum_{j=0}^{N-1} \chi_{U}\circ T^j,
$$
where the cutoff functiion $N$ depends on $U$ and some parameter $t>0$.
$Z^N(x)$ counts the number of entries into $U$ of the point $x$ along the orbit segment 
of length $N$. If we denote by
$$
\tau_U(x)=\inf\{j\ge1: T^jx\in U\}
$$
the entry/return time for the first hit of $U$, then to say that $\tau_U(x)\ge N$ is equivalent to $Z_U^N(x)=0$. 
By Kac's theorem one has for ergodic measures $\mu$ that $\int_U\tau_U(x)\,d\mu(x)=1$.

We assume that $\Omega$ has a finite or countable partition $\mathcal{A}$.
Then $\mathcal{A}^n=\bigvee_{j=0}^{n-1} T^{-j}\mathcal{A}$ be its $n$-th join.
We assume that the elements of $\mathcal{A}^\infty$ consist of singletons.
We assume that the measure $\mu$ is left $\phi$-mixing w.r.t.\ partition $\mathcal{A}$ if 
there exists a decreasing function $\phi(k)\to 0$ as $k\to\infty$ so that
$$
|\mu(B\cap T^{-n-k}C)-\mu(B)\mu(C)|
\le \phi(k)\mu(B)
$$
for all $B\in \sigma(\mathcal{A}^n)$, $C\in\sigma(\bigcup_{\ell=1}^\infty)$ and all $k\in\mathbb{N}$.
If the right hand side equals $\phi(k)\mu(C)$ then we say $\mu$ is right $\phi$-mixing.
We will state all results here for left $\phi$-mixing, but they also all apply to the 
case when $\mu$ is right $\phi$-mixing. The proofs then require some rather obvious 
modifications.

If $U_n\subset \Omega$ is a sequence of nested positive measure sets whose intersection
$\Gamma=\bigcap_nU_n$ forms a zero measure set, e.g.\ a sigle point, then Kac's theorem 
suggests that as $\mu(U_n)\to0$ the associated return times $\tau_{U_n}$ diverge to infinity
at a rate which on average is of the order $\frac1{\mu(U_n)}$. The limiting set $\Gamma$ however
can display some periodic like behaviour within its neighbourhoods in which case we will
get very short returns that are not scaled as Kac's formula suggest as $\mu(U_n)$ goes to zero.
This implies that we get clusters of returns whose returns have conditional probabilities that are
not affected by the shrinking of the set $U_n$ as $n$ goes to zero.
In order to capture this behaviour, let $L$ a large number and let us put  for $k=1,2,\dots$
$$
\lambda_k(L,U) = \frac{\mu(Z^L_{U}=k)}{\mu(\tau_{U}\leq s)}
$$
for the probability to have a cluster of $k$ hits on orbit segment of length $L$ 
provided we record at least one hit during that time. We then put 
$\lambda_k=\lim_{L\to\infty}\lambda_k(L)$, where  
$\lambda_k(L)= \lim_{n\to\infty}\lambda_k(L,U_n)$  provided that the limit exists.
The next theorem then states that the double limit defining $\lambda_k$ can 
be replaced by a single limit along a suitable sequence $(s_n,U_n$ if the measure 
is $\phi$-mixing.

\begin{theorem} \label{sequence.theorem} 
Assume $\mu$ is $\phi$-mixing and let $s_n\to\infty$ be a sequence and $U_n\subset\Omega$
 such that $s_n\mu(\tau_{U_n}\le s_n)\to0$ as $n\to\infty$.
If $\mu$ is $\phi$-mixing and $\lambda_k\not=0$  then 
$$
\lambda_k=\lim_{n\to\infty}\lambda_k(s_n,U_n).
$$
\end{theorem} 

\begin{lemma} \label{lower.bound}
Assume $\mu$ is $\phi$-mixing and let $\gamma>0$ (small). If $\Delta$ so that 
$\phi(\gamma\Delta)<\gamma/2$ and $\mathbb{P}(Z_{U}^{r\Delta})< \gamma/2$, then
$$
\mathbb{P}(Z_U^{r\Delta}\ge1)\ge \frac12 r^{\gamma_0}\mathbb{P}(Z_U^\Delta\ge1),
$$
where $\gamma_0=\frac{\log(2-\gamma)}{\log(2+\gamma)}\approx 1-\frac\gamma{\log2}$.
\end{lemma}

\begin{proof} If we put $E_\Delta=\{Z_U^\Delta\ge 1\}$ then $T^{-j}E_\Delta\subset E_{(2+\gamma)\Delta}$ 
for $j=0,1,\dots, (1+\gamma)\Delta$. Consequently
$$
E_\Delta\cup T^{-(1+\gamma)\Delta}E_\Delta\subset E_{(2+2\gamma)\Delta}
$$
and thus by the $\phi$-mixing property with a gap of length $\gamma\Delta$:
\begin{eqnarray*}
\mu(E_{(2+\gamma)\Delta})
&\ge&\mu\!\left(E_\Delta\cup T^{-(1+\gamma)\Delta}E_\Delta\right)\\
&=& 2\mu(E_\Delta)-\mu(E_\Delta\cap T^{-(1+\gamma)\Delta}E_\Delta)\\
&\ge& 2\mu(E_\Delta)-\mu(E_\Delta)\!\left(\mu(E_\Delta)+\phi(\gamma\Delta)\right)\\
&\ge& (2-\gamma)\mu(E_\Delta)
\end{eqnarray*}
since by assumption $\mu(E_\Delta)+\phi(\gamma\Delta)<\gamma$.
Iterating this dyadic argument we arrive at
$$
\mu(E_{(2+\gamma)^w\Delta})
\ge (2-\gamma)^w\mu(E_\Delta).
$$
If $\gamma_0=\frac{\log(2-\gamma)}{\log(2+\gamma)}$ (which in linear approximation is $1-\frac\gamma{\log2}$),
 then $(2-\gamma)^w=(2+\gamma)^{w\gamma_0}$ which 
then implies the statement in the lemma where the additional factor $\frac12$ accounts for $r$ being
any integer.
\end{proof}

\begin{lemma} \label{ratio.bound}
Assume $\mu$ is $\phi$-mixing and let $\gamma\in(0,\frac12)$. Then for all $r$, $L$ and $\Delta<L$ satisfying
$\phi(\Delta), \mathbb{P}(Z_U^{rL}\ge1)<\frac\gamma2$ 
one has ($\gamma_0=\frac{\log(2-\gamma)}{\log(2+\gamma)}$)
$$
1\ge\frac{\mathbb{P}(Z_U^{rL}\ge1)}{r\mathbb{P}(Z_U^L\ge1)}
\ge  1-\mathbb{P}(Z_U^{rL}\ge1)-\phi(\Delta)-2\!\left(\frac{L}\Delta\right)^{\gamma_0}.
$$
\end{lemma}

\begin{proof} Let us note that the upper bound is trivial. For the lower bound
we get by invariance of the measure
\begin{eqnarray*}
\mathbb{P}(Z_{U}^{rL}\ge 1)
&=&\sum_{j=1}^{r_n}\mathbb{P}(Z_U^{(j-1)L}\ge1,Z_U^L\circ T^{(j-1)L}\ge1, \tau_{U}\circ T^{jL-1}>(r_n-j)L)\\
&=&\sum_{j=1}^{r}\mathbb{P}(Z_U^L\ge1, \tau_{U}\circ T^{L-1}>(r_n-j)L)\\
&=&\sum_{k=0}^{r-1}\mathbb{P}(Z_U^L\ge1, \tau_{U}\circ T^{L-1}>kL)
\end{eqnarray*}
where by the $\phi$-mixing property for some gap $\Delta<L$
\begin{eqnarray*}
\mathbb{P}(Z_U^L\ge1, \tau_{U}\circ T^{L-1}>kL)
&=&\mathbb{P}(Z_U^L\ge1, \tau_{U}\circ T^{L-1+\Delta}>kL-\Delta)+\mathcal{O}(\mathbb{P}(\tau_{U}<\Delta))\\
&=&\mathbb{P}(Z_U^L\ge1)\!\left(\mathbb{P}( \tau_{U}>kL-\Delta)+\mathcal{O}(\phi(\Delta)\right)+\mathcal{O}(\mathbb{P}(\tau_{U}<\Delta)).
\end{eqnarray*}
Since
$$
\mathbb{P}(\tau_{U}>kL-\Delta)=1-\mathbb{P}(Z_U^{kL-\Delta}\ge1)\ge1-\mathbb{P}(Z_U^{rL}\ge1)
$$
we get
\begin{eqnarray*}
\mathbb{P}(Z_{U}^{rL}\ge 1)
&\ge&r\mathbb{P}(Z_U^L\ge1)\!\left(1-\mathbb{P}(Z_U^{rL}\ge1)-\phi(\Delta)\right)-r\mathbb{P}(\tau_{U}<\Delta)
\end{eqnarray*}
and therefore
$$
\frac{\mathbb{P}(Z_U^{rL}\ge1)}{r\mathbb{P}(Z_U^L\ge1)}
\ge 1-\mathbb{P}(Z_U^{rL}\ge1)-\phi(\Delta)-\frac{\mathbb{P}(Z_U^\Delta\ge1)}{\mathbb{P}(Z_U^L\ge1)}
$$
where by Lemma~\ref{lower.bound} 
$\frac{\mathbb{P}(Z_U^\Delta\ge1)}{\mathbb{P}(Z_U^L\ge1)}< 2\!\left(\frac{L}\Delta\right)^{\gamma_0}$ 
(assuming $L$ is an integer multiple of $\Delta$), where $\gamma_0=\frac{\log(2-\gamma)}{\log(2+\gamma)}$.
\end{proof}

Thus lemma allows us now to improve on Lemma~\ref{lower.bound}.

\begin{lemma} \label{lower.bound2}
Assume $\mu$ is $\phi$-mixing and let $\gamma\in(0,\frac12)$. Then for all $r$, $L$ and $\Delta<L/10$ satisfying
$\phi(\Delta), \mathbb{P}(Z_U^{rL}\ge1)<\frac\gamma2$ 
one has ($\gamma_0=\frac{\log(2-\gamma)}{\log(2+\gamma)}$)
$$
r\mathbb{P}(Z_U^L\ge1)
\le  4\mathbb{P}(Z_U^{rL}\ge1).
$$
\end{lemma}

\begin{proof}
This follows from Lemma~\ref{ratio.bound} since with the choice of $\gamma<\frac12$ one has
 $2(L/\Delta)^{\gamma_0}<2 \cdot10^{-\gamma_0}<\frac14$.
\end{proof}

\begin{lemma}\label{k-ratio}
Assume $\mu$ is $\phi$-mixing and let $\gamma>0$ and $\beta\in(0,1)$. Assume that 
$\phi(\gamma L^\beta), \mathbb{P}(Z_U^{rL}\ge1)<\frac\gamma2$.
Then for any $\beta\in(0,1)$ there exists a constant $C$ so that 
$$
\left|\frac{\mathbb{P}(Z_U^{rL}=k)}{r\mathbb{P}(Z_U^L=k)}-1\right|,\left|\frac{r\mathbb{P}(Z_U^L=k)}{\mathbb{P}(Z_U^{rL}=k)}-1\right|
\le C\!\left(\mathbb{P}(Z_U^{rL}\ge1)+\phi(L^\beta)+L^{-(1-\beta)}\frac1{\lambda_k(rL,U)}\right).
$$
\end{lemma}

\begin{proof}
Put $s=rL$, $F_L=\{Z^L=k\}$ ($Z^L=Z_U^L$) and for $j=0,1,\dots,2r-2$ we put 
$$
G_j=F_s\cap T^{-jL/2}F_L
$$
for the good set where the $k$-cluster is located on an orbit segment of length $L$ beginning
at half integer multiples of its length. The we introduce the bad set 
$B_s=F_s\setminus\bigcup_{j=0}^{2r-2}G_j$ and note that $B_s=\varnothing$ if $k=1$.
From now on we therefore assume that $k\ge2$ and note that if $x\in F_s\setminus\bigcup_{j=0}^{2r-2}G_j$ 
then there exists a smallest 
$j<2r-2$ such that $x\in T^{-jL/2}\{Z^{L/2}\ge1\}$ and  $x\not\in T^{-jL/2}F_L$ which implies
$x\in T^{-(j+2)L/2}\{Z^{2r-j-2)L/2}\ge1\}$.
Hence by the $\phi$-mixing property we get
\begin{eqnarray*}
\mathbb{P}(B_s)
&\le&\sum_{j=0}^{2r-2}\mathbb{P}(Z^{L/2}\circ T^{jL/2}\ge1, Z^{(2r-j-2)L/2}\circ T^{(j+2)L/2}\ge1)\\
&\le&\sum_{j=0}^{2r-2}\mathbb{P}(Z^{L/2}\ge1)\left(\mathbb{P}(Z^{(2r-j-2)L/2}\ge1)+\phi(L/2)\right)\\
&\le&2r\mathbb{P}(Z^{L/2}\ge1)\left(\mathbb{P}(Z^{s}\ge1)+\phi(L/2)\right).
\end{eqnarray*}
Since by Lemma~\ref{lower.bound2}
$$
2r\mathbb{P}(Z^{L/2}\ge1)\le 4\mathbb{P}(Z^s\ge1),
$$
 (provided $\phi(\Delta),\mathbb{P}(Z^\Delta\ge1)<\frac14$) we get
$$
\left|\mathbb{P}(F_s)-\mathbb{P}\!\left(\bigcup_{j=0}^{2r-2}G_j\right)\right|
\le4\mathbb{P}(Z^{s}\ge1)\left(\mathbb{P}(Z^{s}\ge1)+\phi(L/2)\right)
$$
Now notice that 
$$
x\in\bigcup_{j=0}^{2r-2}G_j \Rightarrow
\begin{cases} \mbox{ either } x\in \bigcup_{j=0}^r G_{2j}\\
 \mbox{ or } x\in \bigcup_{j=0}^r G_{2j+1}\\
 \mbox{ or both }
 \end{cases}.
$$
Since $G_j\cap G_i=\varnothing$ for all $|i-j|\ge2$ and $G_j\cap G_{j+1}= F_s\cap T^{-(j+1)L/2}F_{L/2}$
 we get by Bonferroni
 \begin{eqnarray*}
 \mathbb{P}\!\left(\bigcup_{j=0}^{2r-2}G_j\right)
 &=&\sum_{j=0}^{2r-2}\mathbb{P}(G_j)-\sum_{j=0}^{2r-3}\mathbb{P}(G_j\cap G_{j+1})\\
 &=&(2r-1)\mathbb{P}(F_L)-(2r-2)\mathbb(F_{L/2})\\
 &=&r\mathbb{P}(F_L)+(r-1)\!\left(\mathbb{P}(F_L)-2\mathbb{P}(F_{L/2})\right).
 \end{eqnarray*}
 To estimate the second term on the RHS note that 
 \begin{eqnarray*}
 \left|\mathbb{P}(F_L)-\mathbb{P}(F_{(L-\Delta)/2}\cup T^{-(L+\Delta)/2}F_{(L-\Delta)/2})\right|\hspace{-3cm}&&\\
 &\le&\mathbb{P}(Z^\Delta\ge1)+\mathbb{P}(Z^{(L-\Delta)/2}\ge1,Z^{(L-\Delta)/2}\circ T^{(L-\Delta)/2}\ge1)\\
 &\le&\mathbb{P}(Z^\Delta\ge1)+\mathbb{P}(Z^{(L-\Delta)/2}\ge1)\!\left(\mathbb{P}(Z^{(L-\Delta)/2}\ge1)+\phi(\Delta)\right)
 \end{eqnarray*}
 by the $\phi$-mixing property. Then
 \begin{eqnarray*}
 \left|\mathbb{P}(F_{(L-\Delta)/2}\cup T^{-(L+\Delta)/2}F_{(L-\Delta)/2})-2\mathbb{P}(F_{(L-\Delta)/2})\right|\hspace{-3cm}&&\\
 &=&\mathbb{P}(F_{(L-\Delta)/2}\cap T^{-(L+\Delta)/2}F_{(L-\Delta)/2})\\
  &\le&\mathbb{P}(Z^{(L-\Delta)/2}\ge1,  Z^{(L-\Delta)/2}\circ T^{-(L+\Delta)/2}\ge1)\\
  &\le&\mathbb{P}(Z^{(L-\Delta)/2}\ge1)\!\left(\mathbb{P}(  Z^{(L-\Delta)/2}\ge1)+\phi(\Delta)\right)
 \end{eqnarray*}
 again using the mixing property. Since also
 $$
  \left|\mathbb{P}(F_{(L-\Delta)/2})-\mathbb{P}(F_{L})\right|
 \le\mathbb{P}(Z^{\Delta/2}\ge1)
  \le\mathbb{P}(Z^{\Delta}\ge1)
  $$
  we get in totem that
 \begin{eqnarray*}
\left|\mathbb{P}(F_L)-2\mathbb{P}(F_{L/2})\right|
&\le&2 \mathbb{P}(Z^\Delta\ge1) +2\mathbb{P}(Z^{L}\ge1)\!\left(\mathbb{P}( Z^{L}\ge1)+\phi(\Delta)\right).
 \end{eqnarray*}
Now, continuing with Lemma~\ref{lower.bound2} let $\gamma>0$ then
 \begin{eqnarray*}
 \left|\mathbb{P}\!\left(\bigcup_{j=0}^{2r-2}G_j\right)-r\mathbb{P}(F_L)\right|
 &\le& r\mathbb{P}(Z^\Delta\ge1)+r\mathbb{P}(Z^{L/2}\ge1)\!\left(\mathbb{P}(Z^{L/2}\ge1)+\phi(\Delta)\right)\\
  &\le&4 \mathbb{P}(Z^{r\Delta}\ge1)
  +2\mathbb{P}(Z^{s}\ge1)\!\left(\mathbb{P}(Z^{s}\ge1)+\phi(\Delta)\right)
 \end{eqnarray*}
provided $\phi(\gamma\Delta),\mathbb{P}(Z^s\ge1)<\frac\gamma2$.
   We have thus shown that
  $$
  \left|\mathbb{P}(Z^s=k)-r\mathbb{P}(Z^L=k)\right|
  \lesssim\mathbb{P}(Z^s\ge1)\!\left(\mathbb{P}(Z^s\ge1)+\phi(L/2)+\phi(\Delta)+\mathbb{P}(Z^{r\Delta}\ge1)\right),
  $$
  and if we put $\Delta=\frac12L^\beta$ for some $\beta\in(0,1)$ then
  $$
  \mathbb{P}(Z^{r\Delta}\ge1)=\mathbb{P}(Z^{sL^{\beta-1}}\ge1)\lesssim \frac1{L^{1-\beta}}\mathbb{P}(Z^s\ge1).
  $$
  (by  Lemma~\ref{lower.bound2} again) and consequently
   $$
  \left|\mathbb{P}(Z^s=k)-r\mathbb{P}(Z^L=k)\right|
  \lesssim\mathbb{P}(Z^s\ge1)\!\left(\mathbb{P}(Z^s\ge1)+\phi(L^\beta)+\frac1{L^{1-\beta}}\mathbb{P}(Z^{s}\ge1)\right).
  $$
Thus
  $$
  \left|\frac{r\mathbb{P}(Z^L=k)}{\mathbb{P}(Z^s=k)}-1\right|
  \lesssim \mathbb{P}(Z^s\ge1)+\phi(L^\beta)+\frac1{L^{1-\beta}}\frac{\mathbb{P}(Z^s\ge1)}{\mathbb{P}(Z^s=k)}
  $$
  which implies the stated estimate. The other estimate is obtained in the same way with an additional
  application of Lemma~\ref{lower.bound2}
\end{proof}

\begin{proof}[Proof of Theorem~\ref{sequence.theorem}]
Let $s_n\to\infty$ be a sequence so that $\mathbb{P}(Z_{U_n}^{s_n}\ge1)\to 0$ as $n\to\infty$.
Assume the limit
$$
\hat\lambda_k=\lim_{n\to\infty}\hat\lambda_k(n)
$$
exists, where
$$
\hat\lambda_k(n)=\frac{\mathbb{P}(Z_{U_n}^{s_n}=k)}{\mathbb{P}(Z_{U_n}^{s_n}\ge1)}.
$$
Let $\varepsilon>0$, then there exists an $L_0$ so that $|\lambda_k-\lambda_k(L)|<\varepsilon$ 
for all $L\ge L_0$. Moreover for every $L$ there exists and $n_0(L)$ so that 
$|\lambda_k(L)-\lambda_k(L,U_n)|<\varepsilon$ for all $n\ge n_0(L)$. Similarly there exists an
$n_1$ so that $|\hat\lambda_k-\hat\lambda_k(n)|<\varepsilon$ for all $n\ge n_1$. 
Let $L\ge L_0$ and put $n_2=n_0(L)\vee n_1$ and $r_n=s_n/L$ for $n\ge n_2$.

One has by Lemma~\ref{ratio.bound} with $\beta\in(0,1)$ 
$$
\frac{\mathbb{P}(Z_{U_n}^{s_n}\ge1)}{r_n\mathbb{P}(Z_{U_n}^L\ge1)}
\ge 1-\mathbb{P}(Z_{U_n}^{s_n}\ge1)-\phi(\Delta)-\frac{4\Delta}L
\ge 1-\mathbb{P}(Z_{U_n}^{s_n}\ge1)-\phi(L^\beta)-4L^{-(1-\beta)}
$$
where $\Delta=L^\beta$  satisfies $\phi(\gamma\Delta),\mathbb{P}(Z^{s_n}\ge1)<\frac\gamma2$
for $L<\!\!< s_n$ large enough as $\mathbb{P}(Z_{U_n}^{s_n}\ge1)\to0$ by assumption.

Therefore we get by Lemmata~\ref{k-ratio} and~\ref{ratio.bound}
\begin{eqnarray*}
\hat\lambda_k(n)
&=&\frac{\mathbb{P}(Z_{U_n}^{s_n}=k)}{\mathbb{P}(Z_{U_n}^{s_n}\ge1)}\\
&=&\frac{\mathbb{P}(Z_{U_n}^{L}=k)}{\mathbb{P}(Z_{U_n}^{L}\ge1)}\frac{\mathbb{P}(Z_{U_n}^{s_n}=k)}{r_n\mathbb{P}(Z_{U_n}^{L}=k)}
\frac{r_n\mathbb{P}(Z_{U_n}^{L}\ge1)}{\mathbb{P}(Z_{U_n}^{s_n}\ge1)}\\
&=&\lambda_k(L,U_n)
\!\left(1+\mathcal{O}\!\left(\mathbb{P}(Z_{U_n}^{s_n}\ge1)+\phi(L^\beta)+L^{-(1-\beta)}\lambda_k(s_n,U_n)^{-1}\right)\right)\times\\
&&\hspace{6cm}\times\!\left(1+\mathcal{O}\!\left(\mathbb{P}(Z_{U_n}^{s_n}\ge1)+\phi(L^\beta)+L^{-(1-\beta)}\right)\right)\\
&=&\lambda_k(L,U_n)(1+\mathcal{O}(\varepsilon))
\end{eqnarray*}
as $\phi(L^\beta), L^{-(1-\beta)}, \mathbb{P}(Z_{U_n}^{s_n}\ge1)<\varepsilon $ if $L$ and $n$ are large enough.
Thus, as $\varepsilon\to0$, $\hat\lambda_k=\lambda_k$.
\end{proof}

\section{The Compound Poisson Distribution}

Suppose that $P\sim\mathrm{Pois}(\lambda)$ is a Poisson random variable, and that $X_1,X_2,\ldots$ are positive, integer-valued, independent, and identically distributed random variables that are all independent from $P$. The random variable
\begin{equation}\label{compoundpoissonrandvar}
W = \sum_{i=1}^{P}X_i
\end{equation}
is called the {\em compound Poisson random variable} with $P$ and $X_i$s.

Observe from the definition that the probability mass function of the compound Poisson random variable given in (\ref{compoundpoissonrandvar}) is given by
$$
\mu(W=k) = \sum_{i=1}^{k}\mu(P=i)\mu(S_i=k) = e^{-\lambda}\sum_{j=1}^{k} \frac{\lambda^i}{i!}\mu(S_i=k),
$$
where $S_i=\sum_{j=1}^{i}X_j$. If in addition we know that $\mu(X_j=k)=\lambda_k$, then we can write
$$
\mu(S_i=k) 
= \mu\!\left(\sum_{j=1}^i X_j=k\right)
 = \mu\!\left(\bigcup_{\sum_{j=1}^i k_j = k}\bigcap_{j=1}^i \{X_j=k_j\}\right)
 = \sum_{\sum_{j=1}^i k_j = k}\prod_{j=1}^i \lambda_{k_j},
$$
where the last equality follows from independence.

Special distributions for the random variables $X_i$s give rise to special distributions for the 
corresponding compound Poisson random variables. A notable kind is when the $X_i$s are 
geometrically distributed with parameter $\theta\in(0,1]$. In that case, the distribution of $W$ is 
characterized by the probability mass function
$$
\mu(W=k) = e^{-\lambda}\sum_{j=1}^k \theta^{k-j}(1-\theta)^j \frac{\lambda^j}{j!}\binom{k-1}{j-1},
$$
and $\mu(W=0)=e^{-\lambda}$. The random variable with this distribution is called the \emph{P\'olya-Aeppli random variable} or the \emph{geometric Poisson random variable} with parameters $\lambda$ and $\theta$.

Similarly if $B(p,n)$ is a binomially distributed random variable then
$$
W=\sum_{i=1}^B X_i
$$
is compound Binomially distributed. If $(p_j,n_j)$ are such that $n_j\to\infty $ as $j\to\infty$ in such a way 
that $p_jn_j\to t$ for some $t>0$, then the associated compound Binomial random variables $W^j$
converge in distribution to a compound Poisson variable. as the charateristic functiion of $W^j$ is
$(1-p_j(1-\varphi(z)))^{n_j}$ converges to the characteristic function of $W$, which is 
$e^{-(1-\varphi(z))}$ where $\varphi(z)=\sum_{k=0}^\infty \mathbb{P}(X_i=k)z^k$ is the 
characteristic function of the i.i.d.\ random variables $X_i$.

We can now formulate our main theorem.

\begin{theorem}\label{main.theorem}
Let $\mathcal{A}$ be a finite generating partition of $(X,\mu,\mathcal{M},T)$, and let $\mu$ be left $\phi$-mixing with 
$\phi(x)=\mathcal{O}(x^{-p})$ for some positive $p$. Suppose $(U_n)_{n=1}^{\infty}$ is a sequence of nested sets 
such that $U_n\in \sigma(\mathcal{A}^n)$ and $\mu(\bigcap_nU_n)=0$. Let $(s_n)_{n=1}^{\infty}$ be a sequence of positive numbers diverging to infinity in such a way that
$\mathbb{P}(Z_{U_n}^{s_n}\ge1)\to0$ as $n\to\infty$ and $s_n^\eta\mathbb{P}(Z_{U_n}^{s_n}\ge1)\to\infty$
for some $\eta\in(0,1)$.
Assume the limits $\hat\lambda_k=\lim_{n\to\infty}\lambda_k(s_n,U_n)$ exist for $k=1,2,\dots$.

Let $t>0$ and put $N_n=\frac{ts_n}{\mathbb{P}(Z_{U_n}^{s_n}\ge1)}$,
then 
$$
Z_{U_n}^{N_n}=\sum_{j=0}^{N_n-1}\chi_{U_n}\circ T^j
$$
 converges in distribution to a compound Poisson random variable with parameters
$t$ and $\hat\lambda_k$, $k=1,\dots$.
\end{theorem}

\section{The Convolution Formula}\label{convolution}

\begin{lemma}\label{convolution.lemma}
Let $(X,\mu,\mathcal{M},T)$ be a measure-preserving dynamical system, and let $\mathcal{A}$ 
be a generating partition. Let $0<s\leq t$ and $U\in\sigma(\mathcal{A}^n)$. 
Assume $\mu$ is left $\phi$-mixing. 

Then for all $0<\Delta<s/2$:
$$
\label{convolutionlemma}
\bigg|\mu(Z_U^{t+s}=k)-\sum_{j=0}^k\mu(Z_U^t=j)\mu(Z_U^s=k-j)\bigg| < \mu(\tau_U^{k+1}>t-\Delta)(4\mu(\tau_U\le\Delta) + 3\phi(\Delta-n)).
$$
\end{lemma}

This convolution type identity is a generalisation of an identity that was used by Galves and Schmitt~\cite{GS97} 
to find the limiting distribution of the first entry times to cylinder sets centred at points. Their identity
corresponds to the case here when $k=0$ and then it lead to in the limit to an exponential 
distribution.

\begin{proof}
Let $\Delta <s/2$ and assume we cut a cap of length $\Delta$. Then the probability of having exactly $j$ 
returns in the first $t$ amount of time, and then having exactly $k-j$ returns in the timespan $[t+\Delta,s-\Delta]$
is given by
$$
\sum_{j=0}^k\mu(\{Z_U^t=j\}\cap\{Z_U^{s-\Delta}\circ T^{t+\Delta}=k-j\}).
$$
Now put
$$
E= \{Z_U^t+Z_U^{s-\Delta}\circ T^{t+\Delta}=k\}
$$
Partitioning this set $E$ by the number of returns in the gap $[t,t+\Delta]$ yields
\begin{eqnarray*}
&\bigcup_{j=0}^k\{Z_U^t=j\}\cap\{Z_U^{s-\Delta}\circ T^{t+\Delta}=k-j\} = (E\cap \{\tau_U^1\circ T^t>\Delta\})\cup (E\cap \{\tau_U^1\circ T^t\leq \Delta\})
\end{eqnarray*}
where set has no returns in the gap, and the other set has least one return in the gap. 
Observe that the former is a subset of $\{Z_U^{t+s}=k\}$, since if no return occurs in addition to the already existing 
$k$ returns, we have exactly $k$ returns. Thus
\begin{eqnarray*}
\abs{\mu(Z_U^{t+s}=k)-\mu(E)}
& \leq& \mu(\{Z_U^{t+s}=k\}\cap\{\tau\circ T^t<\Delta\}\cup E\cap\{\tau\circ T^t<\Delta\})\\
& \leq &\mu(\{Z_U^{t+s}=k\}\cap \{\tau_U\circ T^t\le\Delta\}) +\mu(E\cap \{\tau_U\circ T^t\leq \Delta\}).
\end{eqnarray*}
The first term, $\mu(\{Z_U^{t+s}=k\}\cap \{\tau_U\circ T^t\leq\Delta\})$, can be approximated as follows:
\begin{eqnarray*}
\mu(\{Z_U^{t+s}=k\}\cap \{\tau_U\circ T^t\leq\Delta\})
& \leq& \mu(\{Z_U^{t-\Delta}\leq k\}\cap \{\tau_U\circ T^t\leq\Delta\}) \\
& =& \sum_{j=0}^k\mu(\{Z_U^{t-\Delta}=j\}\cap\{\tau_U\circ T^t\leq\Delta\}).
\end{eqnarray*}
That is, if the total number of returns add up to $k$, then it must be that during the first $t-\Delta$ amount of time, 
no more than $k$ returns should occur. Since $\{Z_U^{t-\Delta}=j\}$ is in 
$\sigma((\mathcal{A}^n)^{t-\Delta})=\sigma(\mathcal{A}^{n+t-\Delta})$, and since for the event $\{\tau_U\leq\Delta\}$, we have
$$
\{\tau_U\circ T^t<\Delta\} = T^{-t}\{\tau_U<\Delta\} = T^{-n-t+\Delta+n-\Delta}\{\tau_U<\Delta\},
$$
we can estimate the first term using the $\phi$-mixing property:
\begin{eqnarray*}
\sum_{j=0}^k\mu(\{Z_U^{t-\Delta}=j\}\cap\{\tau_U\circ T^t\leq\Delta\})
& \leq& \sum_{j=0}^k\mu(\{Z_U^{t-\Delta}=j\}\cap T^{-t}\{\tau_U\leq\Delta\}) \\
& \leq &\sum_{j=0}^k\mu(Z_U^{t-\Delta}=j)(\mu(\tau_U\leq\Delta)+\phi(\Delta-n)) \\
& = &\mu(\tau_U^{k+1}>t-\Delta)(\mu(\tau_U\leq\Delta)+\phi(\Delta-n)).
\end{eqnarray*}
The rest is estimated very crudely:
$$
E = \{Z_U^t+Z_U^{s-\Delta}\circ T^{t+\Delta}=k\} 
\subset \{Z_U^{t-\Delta}\leq k\} 
= \bigcup_{j=0}^k \{Z_U^{t-\Delta}=j\},
$$
which lets us conclude that
\begin{eqnarray*}
\mu(E\cap \{\tau_U\circ T^t\leq \Delta\})
& \leq &\sum_{j=0}^k \mu(\{Z_U^{t-\Delta}=j\}\cap \{\tau_U\circ T^t\leq \Delta\})\\
& \leq &\mu(\tau_U^{k+1}>t-\Delta)(\mu(\tau_U\leq\Delta)+\phi(\Delta-n)).
\end{eqnarray*}
We get for the middle error term
combined with the fact that $\{Z_U^t=j\}\in \sigma((\mathcal{A}^n)^t)=\sigma(\mathcal{A}^{n+t})$
$$
\abs{\mu(\{Z_U^t=j\}\cap\{Z_U^{s-\Delta}\circ T^{t+\Delta}=k-j\})-\mu(Z_U^t=j)\mu(Z_U^{s-\Delta}=k-j)} \leq \mu(Z_U^t=j)\phi(\Delta-n),
$$
using the $\phi$-mixing property. Hence when summed over $j$ from 0 to $k$, we obtain
$$
\sum_{j=0}^k \mu(Z_U^t=j)\phi(\Delta-n) = \mu(\tau_U^{k+1}>t)\phi(\Delta-n) \leq \mu(\tau_U^{k+1}>t-\Delta)\phi(\Delta-n).
$$
Finally, observe that the last part
\begin{eqnarray*}
&&\bigg|\sum_{j=0}^k\mu(Z_U^t=j)\mu(Z_U^{s-\Delta}=k-j)-\sum_{j=0}^k\mu(Z_U^t=j)\mu(Z_U^{s}=k-j)\bigg|\\
&&\hspace{4cm}= \sum_{j=0}^k\mu(Z_U^t=j)\abs{\mu(Z_U^{s}=k-j)-\mu(Z_U^{s-\Delta}=k-j)}
\end{eqnarray*}
can be approximated by making the observation similar to that we made for the first error term. Namely, we begin by noticing that
$$
\abs{\mu(Z_U^{s}=k-j)-\mu(Z_U^{s-\Delta}=k-j)} \leq \mu(\{Z_U^s=k-j\}\mathbin{\triangle}\{Z_U^{s-\Delta}=k-j\}).
$$
Observe that the intersection of $\{Z_U^s=k-j\}$ and $\{Z_U^{s-\Delta}=k-j\}$ is precisely the set 
$\{Z_U^{s-\Delta}=k-j\}\cap\{\tau_U\circ T^{s-\Delta}>\Delta\}$, therefore
$$
\{Z_U^{s-\Delta}=k-j\}\cap\{\tau_U\circ T^{s-\Delta}>\Delta\}\subset \{Z_U^s=k-j\},
$$
meaning one part of the symmetric difference is
$$
\{Z_U^s=k-j\}\setminus(\{Z_U^{s-\Delta}=k-j\}\cap\{\tau_U\circ T^{s-\Delta}>\Delta\}) = \{Z_U^s=k-j\}\cap\{\tau_U\circ T^{s-\Delta}\leq \Delta\}.
$$
Hence by monotonicity and by the fact that $T$ is measure-preserving,
\begin{eqnarray*}
\mu(\{Z_U^s=k-j\}\cap\{\tau_U\circ T^{s-\Delta}\leq \Delta\})
& \leq &\mu(\tau_U\circ T^{s-\Delta}\leq \Delta) \\& = &\mu(\tau_U\leq \Delta).
\end{eqnarray*}
For the other part of the symmetric difference, by the same token,
\begin{eqnarray*}
\mu(\{Z_U^{s-\Delta}=k-j\}\cap\{\tau_U\circ T^{s-\Delta}\leq\Delta\}) \leq \mu(\tau_U\leq \Delta).
\end{eqnarray*}
Hence the third error term is bounded by
\begin{eqnarray*}
\sum_{j=0}^k\mu(Z_U^t=j)\abs{\mu(Z_U^{s}=k-j)-\mu(Z_U^{s-\Delta}=k-j)}
& \leq &\sum_{j=0}^k2\mu(Z_U^t=j)\mu(\tau_U\leq\Delta) \\
& =& 2\mu(\tau_U^{k+1}>t)\mu(\tau_U\leq\Delta) \\
& \leq &2\mu(\tau_U^{k+1}>t-\Delta)\mu(\tau_U\leq \Delta).
\end{eqnarray*}
Combined, this gives the final error estimate
\begin{eqnarray*}
&&\bigg|\mu(Z_U^{t+s}=k)-\sum_{j=0}^k\mu(Z_U^t=j)\mu(Z_U^s=k-j)\bigg| \\
&&\qquad< \mu(\tau_U^{k+1}>t-\Delta)(4\mu(\tau_U\leq\Delta) + 3\phi(\Delta-n)).
\end{eqnarray*}
\end{proof}

For the first entry time we get the well-known and much used special case
$$
\abs{\mu(\tau_U>s+t)-\mu(\tau_U>t)\mu(\tau_U>s)}\leq \mu(\tau_U>t-\Delta)(4\Delta\mu(U)+2\phi(\Delta-n)).
$$
which was first derived by Galves and Schmitt~\cite{GS97} in the $\psi$-mixing case to get the 
limiting distribution of entry times. Here we used the estimate
$\mu(\tau_U\le \Delta)\le\sum_{j=0}^{\Delta-1}\mu(T^{-j}U)=\Delta\mu(U)$.


\section{Generating Functions}\label{generating.function}

Finding the limiting distribution of $Z^{N_n}_{U_n}$ is tantamount to finding the limit 
$\mu_k^{N_n}= \mu(Z^{N_n}_{U_n}=k)$ for every $k\in\mathbb{N}_0$. 
For this we shall use generating functions to approximate the distribution of $Z^{N_n}_{U_n}$ 
by a compound Binomial distribution for which Lemma~\ref{convolutionlemma} will be used.

Let $U$ be a measurable set and denote by $\mu^s_k$ the measure $\mu(Z_U^s=k)$. Since most of the times it will be evident what $U$ we are referring to (which, later, will be an element of a sequence $(U_n)_{n=1}^{\infty}$), we will suppress it and will simply write $\mu_k^s$. For $s$ fixed, consider the probability generating function
$$
F_s(z) = \sum_{k=0}^{\infty}\mu_k^sz^k.
$$
Since $\abs{\mu_k^s}\leq 1$ for every $k$, $F_s(z)$ is an analytic function in the unit disk.
 Observe further that $F_s(0)=\mu_0^s = \mu(\tau_U^1\ge s)$, and that
$$
\frac{1}{\ell!}\frac{d^\ell}{dz^\ell}F_s(z)\bigg|_{z=0} = \mu_\ell^s.
$$
$$
\eta(\Delta) = 4\Delta\mu(\tau_U\Delta)+3\phi(\Delta-n).
$$
Recall that by Lemma~\ref{convolution.lemma}, we have
\begin{eqnarray*}
\bigg|\mu_k^{t+s}-\sum_{j=0}^k\mu_j^t\mu_{k-j}^s\bigg| 
& < &\mu(\tau_U^{k+1}>t-\Delta)(4\Delta\mu(U)+3\phi(\Delta-n))\\
& < &\eta(\Delta)\sum_{j=0}^{k}\mu_j^{t-\Delta}\\
& \leq &\eta(\Delta).
\end{eqnarray*}
Assume that $\Delta$ is a function of $s$, so $\Delta=\Delta(s)$, whereupon $\eta$ implicitly depends on $s$. Denote by $\tilde{\eta}$ the composition $\eta\circ \Delta$, so that $\tilde{\eta}(s) = \eta(\Delta(s))$. Fix $s$ and for $a=1,2,\dots$, define 
$$
\xi_k^{as}= \sum_{j=0}^{k}\mu_j^{(a-1)s}\mu_{k-j}^s -\mu_k^{as},
$$
where $\xi_k^s=0$ and arrive at
$$
\abs{\xi^{as}_k}  \leq \mu(\tau_U^{k+1}>as)\tilde{\eta}(s) = \tilde{\eta}(s)\sum_{j=0}^{k}\mu_{j}^{as}.
$$
For $n=2,\dots$, put
$$
G_n(z)= \sum_{k=0}^{\infty}\xi_k^{(n-1)s}z^k.
$$
and $G_1=0$.

\begin{lemma}\label{inductionlemma}
For every positive integer $r>1$, on $\abs{z}<1$, the analytic function $F_s(z)$ satisfies
$$
F_{s}(z)^r = F_{rs}(z) + \sum_{k=2}^{r} G_k(z)F_s(z)^{r-k}.
$$
\end{lemma}

\begin{proof}
For $r=2$ we get
\begin{eqnarray*}
F_s(z)^2 
& =& \sum_{k=0}^{\infty}\mu_k^sz^k\sum_{k=0}^{\infty}\mu_k^sz^k \\
& =& \sum_{k=0}^{\infty}z^k\sum_{j=0}^k \mu_{j}^s\mu_{k-j}^s \\
& = &\sum_{k=0}^{\infty}(\mu_k^{2s}+\xi_k^s)z^k \\
& =& F_{2s}(z)+ G_2(z),
\end{eqnarray*}
where $G_2(z)= \sum_{k=0}^{\infty}\xi_k^sz^k$. 
For the induction step we get
$$
F_s(z)^{r+1} = F_s(z)^rF_s(z) = F_{rs}(z)F_s(z) + \sum_{k=2}^{r} G_k(z)F_s(z)^{r+1-k}
$$
where the term $F_{rs}(z)F_s(z)$ gives
\begin{eqnarray*}
F_{rs}(z)F_s(z)
& =& \sum_{k=0}^{\infty}\mu^{rs}_kz^k\sum_{k=0}^{\infty}\mu_k^sz^k \\
& = &\sum_{k=0}^{\infty}z^k\sum_{j=0}^k \mu_j^{rs}\mu_{k-j}^s \\
& =& \sum_{k=0}^{\infty} (\mu_k^{(r+1)s}+\xi^{rs}_k)z^k \\
& =& F_{(r+1)s}(z) + G_{r+1}(z)
\end{eqnarray*}
which proves the formula claimed.
\end{proof}

Our objective is to take the $k$-th derivative of $F_{rs}(z)$ with respect to the complex variable $z$, evaluate it at $z=0$, 
and obtain the formula for $\mu_k^{rs}$ for large $r$ and $s$. Let us denote by $E_s^r(z)$ the error function
$$
E_s^r(z) =  \sum_{k=2}^{r} G_k(z)F_s(z)^{r-k},
$$
so that 
$$
F_{rs}(z) = F_s(z)^r-E_s^r(z).
$$

\begin{lemma}\label{derivative.lemma}
Then there exists a constant $C$ so that 
$$
\bigg|\frac{1}{k!}\frac{d^k}{dz^k}E_s^r(0)\bigg| \le C kr\tilde{\eta}(s).
$$
\end{lemma}

\begin{proof}
If $0<\rho<1$ then Cauchy's estimate on the disk $\abs{z}\leq\rho$
\begin{equation}\label{cauchy}
\bigg|\frac{d^k}{dz^k}E_s^r(0)\bigg| \leq \frac{k!M}{\rho^k},
\end{equation}
where
$$
M = \sup_{z:\abs{z}=\rho}\abs{E_s^r(z)}.
$$
Since the power series representation of $F_s(z)$
 has only real coefficients $\mu_k^s$ (that sum to 1), we have $\abs{F_s(z)}\leq F_s(\abs{z})$. 
Therefore, for  $\abs{z}\leq\rho<1$ we get
\begin{eqnarray*}
F_s(\abs{z})
 &\leq& \mu_0^s + \sum_{k=1}^{\infty}\mu_k^s\abs{z}^k \\
& \leq& \mu_0^s + \sum_{k=1}^{\infty}\mu_k^s\rho^k \\
& \leq& \mu_0^s + \sum_{k=1}^{\infty}\mu_k^s\rho \\
& =& \mu(\tau_U^1>s)+\rho\mu(\tau_U^1\leq s) \\
&<&\mu(\tau_U^1>s)+\mu(\tau_U^1\leq s)=1,
\end{eqnarray*}
unless $\mu_0^s=1$ (then $F_s(z)\equiv 1$).
Moreover we have
$$
\abs{G_n(z)}
 \leq \sum_{k=0}^{\infty}\abs{\xi_k^{(n-1)s}}\abs{z}^k
  \leq \sum_{k=0}^{\infty} \tilde{\eta}(s)\abs{z}^k 
  = \frac{\tilde{\eta}(s)}{1-\abs{z}},
$$
as $\abs{\xi^{as}_k} \le \mu(\tau_U^{k+1}>as)\tilde{\eta}(s)$.
provided that $\abs{z}<1$. 
Consequently
$$
\abs{E_s^r(z)} \leq \sum_{k=2}^{r}\abs{G_k(z)} \leq \frac{\tilde{\eta}(s)(r-1)}{1-\rho}.
$$
So in this case, the total error in the $k$-th derivative is bounded by
$$
\bigg|\frac{d^k}{dz^k}E_s^r(0)\bigg| \leq \frac{k!\tilde{\eta}(s)(r-1)}{(1-\rho)\rho^k}.
$$
In all other cases, put $\sigma= \mu(\tau_U>s)+\rho\mu(\tau_U\leq s)<1$. Then
$$
\abs{E_s^r(z)} 
 \leq \sum_{k=2}^{r}\abs{G_k(z)}\abs{F_s(z)}^{r-k} 
 \leq \frac{\tilde{\eta}(s)}{1-\rho}\sum_{k=2}^{r}F_s(\abs{z})^{r-k} 
 \leq \frac{\tilde{\eta}(s)}{1-\rho}\sum_{k=2}^{r}\sigma^{r-k} 
 <  \frac{\tilde{\eta}(s)(r-1)}{(1-\rho)}.
$$
Hence the total error in the $k$-th derivative is bounded by the same limiting term
as $\sigma\le1$:
$$
\bigg|\frac{d^k}{dz^k}E_s^r(0)\bigg| \leq \frac{k!\tilde{\eta}(s)(r-1)}{(1-\rho)\rho^k}.
$$

The factor $(1-\rho)^{-1}\rho^{-k}$ attains the minimum on $(0,1)$ at $\rho=\frac{k}{1+k}$ which implies
 $(1-\rho)^{-1}\rho^{-k}=k^{-k}(1+k)^{1+k}$ and therefore
\begin{eqnarray*}
\bigg|\frac{d^k}{dz^k}E_s^r(0)\bigg| 
&\leq& \tilde{\eta}(s)(r-1)k!k^{-k}(1+k)^{1+k}\\
&\leq& \tilde{\eta}(s)(r-1)k!(1+k)(1+k^{-1})^{k}\\
&\leq& e\tilde{\eta}(s)(r-1)(k+1)!
\end{eqnarray*}
as $(1+k^{-1})^k<e$. Hence we deduce that
$$
\bigg|\frac{1}{k!}\frac{d^k}{dz^k}E_s^r(0)\bigg| \leq e\tilde{\eta}(r-1)(s)(k+1) < er\tilde{\eta}(s)(k+1),
$$
which implies
$$
\bigg|\frac{1}{k!}\frac{d^k}{dz^k}E_s^r(0)\bigg| \lesssim k r\tilde{\eta}(s).
$$
\end{proof}

\begin{proof}[Proof of Theorem~\ref{main.theorem}]
Let $t>0$ a parameter let $s_n\to \infty$ be a sequence so that $\mathbb{P}(Z_{U_n}^{s_n}\ge1)\to 0$ 
as $n\to\infty$ and put $r_n=\frac{t}{\mathbb{P}(Z_{U_n}^{s_n}\ge1)}$. Then we define the 
observation time by $N_n=s_nr_n$. 

Let $\tilde{Z}_{U_n}^{s_n,j}$, $j=0,1,2,\dots,r_n-1$, be i.i.d.\ random variables which have the same distributions 
as $Z_{U_n}^{s_n}$. Then then the random variable $\tilde{W}^n=\sum_{j=0}^{r_n-1}\tilde{Z}_{U_n}^{s_n,j}$
 is compound Binomially distributed with $p_n=\mathbb{P}(Z_{U_n}^{s_n}\ge1)$ and $r_n$ 
 and has parameters
$$
\hat\lambda_k(n)=\lambda_k(s_n,U_n)=\frac{\mathbb{P}(Z_{U_n}^{s_n}=k)}{\mathbb{P}(Z_{U_n}^{s_n}\ge1)}
$$
 for $k=1,2,\dots$.
If $\tilde{F}_n(z)=\sum_{k=0}^\infty\tilde\mu_k^{r_n,s_n}$ is the generating function of $\tilde{W}^n$, 
where $\tilde{\mu}_k^{r_n,s_n}=\mathbb{P}(\tilde{W}^n=k)$, then $\tilde{F}_n(z)=F_{s_n}(z)^{r_n}$
and by Lemma~\ref{derivative.lemma} then get the estimate
$$
\left|\mu^{r_ns_s}_k-\frac{1}{k!}\frac{d^k}{dz^k}F_{s_n}(z)^{r_n}\right|_{z=0}\bigg| \lesssim k r_n\tilde{\eta}(s_n).
$$
and using Lemma~\ref{lower.bound}
$$
\left|\mu_k^{N_n}-\tilde\mu_k^{r_n,s_s}\right|
\lesssim r_n\!\left(\mathbb{P}(Z_{U_n}^{\Delta_n}\ge1)+\phi(\Delta_n)\right)
\lesssim \mathbb{P}(Z_{U_n}^{r_n\Delta_n}\ge1)+r_n\phi(\Delta_n).
$$
For $\alpha\in(0,1)$ we put $\Delta_n=s_n^\beta$ and obtain that
$r_n\Delta_n=\frac{t}{\mathbb{P}(Z_{U_n}^{s_n}\ge1)}s_n^\alpha=\frac{N_n}{s_n^{1-\alpha}}$.
If $\phi$ is polynomially decreasing at rate $p$, i.e.\  $\phi(\ell)\lesssim \ell^{-p}$, 
then 
$$
\left|\mu_k^{N_n}-\tilde\mu_k^{r_n,s_n}\right|
\lesssim \mathbb{P}(Z_{U_n}^{r_n\Delta_n}\ge1)+\frac{s_n^{-p\alpha}}{\mathbb{P}(Z_{U_n}^{s_n}\ge1)}
$$
where the first term goes to zero if we can assure that 
$\mathbb{P}(Z_{U_n}^{r_n\Delta_n}\ge1)\le \mathbb{P}(Z_{U_n}^{s_n}\ge1)$.
To achieve this we require that $r_n\Delta_n<s_n$ which is implied by
$s_n^{1-\alpha}\mathbb{P}(Z_{U_n}^{s_n}\ge1)\to\infty$.
The  second term goes to zero if
$s_n^{p\alpha}\mathbb{P}(Z_{U_n}^{s_n}\ge1)\to \infty$ as $n\to\infty$.
All those conditions can be satisfied with suitable choices of $\gamma$ and $\alpha$ is
$s_n^\eta\mathbb{P}(Z_{U_n}^{s_n}\ge1)\to \infty$ as $n\to\infty$ for some $\eta\in(0,1)$.
We thus have that for every $k\in\mathbb{N}_0$:
$$
\left|\mathbb{P}(Z_{U_n}^{N_n}=k)-\mathbb{P}(\tilde{W}^n=k)\right|
\to 0
$$ 
as $n\to\infty$ which means that the distributional distance between $Z_{U_n}^{N_n}$ 
and $\tilde{W}^n$ goes to zero. 
Since $r_np_n=t$ one gets that $\tilde{W}^n$ converges in distribution to the compound Poisson $W$
distribution with parameters $t$ and $\lambda_k$. where $\lambda_k=\hat\lambda_k$ 
by Theorem~\ref{sequence.theorem}. Hence
$Z_{U_n}^{s_n}\longrightarrow W$ in distribution where $W$ is the compound distribution 
with parameters $t$. and $\lambda_k$, $k=1,2,\dots$.
\end{proof}

\section{Remarks}

\subsection{Kac scaling}

Let $(U_n)_{n=1}^{\infty}$ be a nested sequence of measurable sets such that 
$\bigcap_{n=1}^{\infty}U_n$ is a null set. Then for  $k\ge1$ put
 $$
 \alpha_k=\lim_{L\to\infty}\lim_{n\to\infty}\mathbb{P}(\tau_{U_n}^{k-1}< L\le\tau_{U_n}^k|U_n)
 $$
 where in particular if $k=1$ the coefficient $ \alpha_1=\lim_{L\to\infty}\lim_{n\to\infty}\mathbb{P}( L\le\tau_{U_n}|U_n)$
 is the extremal index.
 
\begin{lemma} 
Let $\mu$ be a $T$-invariant probability measures. Then
$$
\mathbb{P}(\tau_U<L)=\sum_{k=1}^{L}\mathbb{P}(U, \tau_{U}\ge k).
$$
If moreover the limit $\alpha_1=\lim_{L\to\infty}\lim_{n\to\infty}\alpha_1(L,U_n)$ exists and is positive
for a nested sequence $U_n$ so that $\mu(\bigcap_n U_n)=0$, then
$$
\lim_{L\to\infty}\lim_{n\to\infty}\frac{\mathbb{P}(\tau_{U_n}< L)}{L\mu(U_n)}=\alpha_1.
$$
\end{lemma}

\begin{proof} By invariance of the measure
\begin{eqnarray*}
\mathbb{P}(\tau_U<L)&=&\mathbb{P}(Z_{U}^{L}\ge 1)\\
&=&\sum_{j=0}^{L-1}\mathbb{P}(Z_U^{j}\ge1,T^{-j}U, \tau_{U}\circ T^{j}\ge L-j)\\
&=&\sum_{j=0}^{L-1}\mathbb{P}(U, \tau_{U}\ge L-j)\\
&=&\sum_{k=1}^{L}\mathbb{P}(U, \tau_{U}\ge k).
\end{eqnarray*}
To prove the second statement we see that $\mathbb{P}(\tau_{U_n}<L)=\sum_{k=1}^L\mu(U_n)\alpha_1(L,U_n)$
and taking limits provides the second statement.
\end{proof}

If $\alpha_1>0$ then by~\cite{HV09} Theorem~2
$$
\lambda_k = \frac{\alpha_k-\alpha_{k+1}}{\alpha_1},
$$
provided  $\sum_{k=1}^{\infty}k^2{\alpha}_k<\infty$.
Similarly, again, if $\alpha_1\not=0$, one gets
$$
\lim_{s\to\infty}\lim_{n\to\infty}\frac{\mu(Z^s_{U_n}=k)}{s\mu(U_n)} = \alpha_1\lambda_k.
$$

If the extremal index $\alpha_1$ is strictly positive then the scaling will be the traditional Kac scaling 
with a speed adjustment given by $\alpha_1$. 
Indeed, by Theorem~\ref{sequence.theorem} the observation time is then
$N_n=\frac{t}{\alpha_1\mu(U_n)}$ where we have used that $\frac{\mathbb{P}(Z_{U_n}^{s_n}\ge1)}{\alpha_1\mu(U_n)}$
converges to $1$ by Lemma~\ref{ratio.bound} as $n\to\infty$. If $\mu$ is $\phi$-mixing at a rate $\phi(\ell)\lesssim \ell^{-p}$
for some $p>0$ let us choose  $\omega\in(\frac1{1+p},1)$.  Then we put $s_n=\mu(U_n)^{-\omega}$ and $\Delta_n=s_n^\alpha$,
where we can choose $\alpha\in(0,1)$ so that $\frac1{a+\alpha p}<\omega$.
Then $s_n\mu(U_n)=\mu(U_n)^{1-\omega}\to0$ and $s_n^{1+p\alpha}\mu(U_n)=\mu(U_n)^{1-\omega(1+\alpha p)}\to\infty$ 
as the exponent $1-\omega(1+\alpha p)$ is negative. 
Let us observe that $r_n\mathbb{P}(Z_{U_n}^{\Delta_n}\ge1)\le r_n\Delta_n\mu(U_n)$ 
and therefore 
\begin{eqnarray*}
\left|\mathbb{P}(Z_{U_n}^{N_n}=k)-\mathbb{P}(\tilde{W}^n=k)\right|
&\lesssim &r_n\Delta_n\mu(U_n) +\frac{s_n^{-\alpha p}}{s_n\mu(U_n)}\\
&\lesssim& \mu(U_n)^{\omega+\alpha\omega} +\mu(U_n)^{\omega(1+\alpha)-1}
\end{eqnarray*}
which goes to zero as $n\to\infty$ as both exponents are positive.
We thus have proven the following corollary.

\begin{corollary}\label{corollary.kac}
Let $\mu$ be left $\phi$-mixing with 
$\phi(x)=\mathcal{O}(x^{-p})$, $p>0$. Let $(U_n)_{n=1}^{\infty}$ be a sequence of nested sets 
satisfying $U_n\in \sigma(\mathcal{A}^n)$ and $\mu(\bigcap_nU_n)=0$.
Assume the limits $\lambda_k=\lim_{L\to\infty}\lim_{n\to\infty}\lambda_k(L,U_n)$ exist for $k=1,2,\dots$
and put $\alpha_1=\left(\sum_{k=1}^\infty\lambda_k\right)^{-1}$.

Let $t>0$ and put $N_n=\frac{t}{\alpha_1\mu(U_n)}$,
then 
$$
Z_{U_n}^{N_n}=\sum_{j=0}^{N_n-1}\chi_{U_n}\circ T^j
$$
 converges in distribution to a compound Poisson random variable with parameters
$t$ and $\lambda_k$, $k=1,\dots$.
\end{corollary}

\subsection{$\alpha$-mixing measures}

The measure $\mu$ is $\alpha$-mixing if 
$$
|\mu(B\cap T^{-n-k}C)-\mu(B)\mu(C)|
\le  \alpha(k)
$$
for all $B\in \sigma(\mathcal{A}^n)$, $C\in\sigma(\bigcup_{\ell=1}^\infty)$ and all $k\in\mathbb{N}$,
where $\alpha(k)\to 0$ as $k\to\infty$. From the proof of Lemma~\ref{convolution.lemma}
we get the adjusted statement that for all $\Delta<s/2$ now reads
$$
\label{convolutionlemma}
\bigg|\mu(Z_U^{t+s}=k)-\sum_{j=0}^k\mu(Z_U^t=j)\mu(Z_U^s=k-j)\bigg| 
\lesssim \mu(\tau_U\le\Delta) + \alpha(\Delta)
$$
(where we use that $\alpha(\Delta-n)\lesssim\alpha(\Delta)$ for all $U\in \sigma(\mathcal{A}^n)$
and $s<t$. Consequently, if $U_n\in\sigma(\mathcal{A}^n)$ is a nested sequence with $\mu(U_n)\to0$,
then for numbers $s_n\to\infty$ we put as before $r_n=\frac{t}{\mathbb{P}(Z_{U_n}^{s_n}\ge1)}$
and $N_n=r_ns_n$ where $t>0$ is a parameter, we get
$$
\left|\mu_k^{N_n}-\tilde\mu_k^{r_n,s_s}\right|
\lesssim r_n\!\left(\mathbb{P}(Z_{U_n}^{\Delta_n}\ge1)+\alpha(\Delta_n)\right).
$$
Since we cannot use Theorem~\ref{sequence.theorem} and Lemma~\ref{lower.bound} we can only formulate
the following result whose main assumption is that the cluster probabilites $\lambda_k$ can be 
realised by single limit along a suitable sequence $(s_n,U_n)$.

\begin{theorem}\label{alpha.theorem}
Let  $\mu$ be an $\alpha$-mixing $T$-invariant probability measure where 
$\alpha(x)=\mathcal{O}(x^{-p})$ for some positive $p$. Let $(U_n)_{n=1}^{\infty}$ is a sequence of nested sets 
such that $U_n\in \sigma(\mathcal{A}^n)$ and $\mu(\bigcap_nU_n)=0$. 

Let $\alpha\in(0,1)$ and $t>0$ and assume $s_n\to\infty$ is a sequence of numbers so that
if we put $\Delta_n=s_n^\alpha$, $r_n=\frac{t}{\mathbb{P}(Z_{U_n}^{s_n}\ge1)}$ and $N_n=r_ns_n$ 
then one has
$\mathbb{P}(Z_{U_n}^{s_n}\ge1)\to0$ as $n\to\infty$, $s_n^\eta\mathbb{P}(Z_{U_n}^{s_n}\ge1)\to\infty$
for some $\eta\in(0,1)$ and $r_n\mathbb{P}(Z_{U_n}^{\Delta_n}\ge1)\lesssim\mathbb{P}(Z_{U_n}^{s_n}\ge1)$.

If  $\lambda_k=\lim_{n\to\infty}\lambda_k(s_n,U_n)$  for $k=1,2,\dots$, then
$$
Z_{U_n}^{N_n}=\sum_{j=0}^{N_n-1}\chi_{U_n}\circ T^j
$$
 converges in distribution to a compound Poisson random variable with parameters
$t$ and $\lambda_k$, $k=1,\dots$.
\end{theorem}

\subsection{Gibbs Markov measures}
Let $(\Omega,\mu)$ be Lebesgue space with a map $T:\Omega\to\Omega$ and 
$\mu$  $T$-invariant probability measure. Let $\mathcal{A}$  be a countable partition of
$\Omega$ which we assume has the Markov property, that is for every $A\in\mathcal{A}$
the image $T(A)$ is a union of elements in $\mathcal{A}$ and that moreover $T$ is one-to-one
on each element in $\mathcal{A}$. We also assume that the partition is generating, i.e.\
$\mathcal{A}^\infty$ consists of singletons. If one puts $\varphi=\frac{d\mu}{d(\mu\circ T)}$ for 
the potential function of $\mu$ then we say the map $T$ is {\em Gibbs-Markov} if:\\
(i) (Big image property) There exists a constant $C>0$ such that $\mu(A)\ge C$ for all $A\in\mathcal{A}$.\\
(ii) (Distortion) $\log \varphi$ is Lipschitz continuous on each element in $\mathcal{A}$ with respect 
to the separation metric $d_\vartheta$.\\
The separation metric is given by $d_\vartheta(x,y)=\vartheta^{s((x,y)}$ for some $\vartheta\in(0,1)$ 
and where $s(x,y)=\min\{j\ge0: A(T^jx)\not=A(T^jy)\}$ and $A(z)\in\mathcal{A}$ so that $z\in A(z)$.
For such measure $\mu$ it was shown in~\cite{MN05}, Lemma~2.4, that there exists a $\tau\in(0,1)$
and a constant $c_1$ so that 
$$
\left|\mu(B\cap T^{-k-n}C)-\mu(B)\mu(C)\right|
\le c_1\tau^k\mu(B)\sqrt{\mu(C)}
$$
for all $B\in\sigma(\mathcal{A}^n)$, $C\in\sigma(\bigcup_{\ell=1}^\infty\mathcal{A}^\ell)$ and all $k\in\mathbb{N}$.
This implies that $\mu$ is $\phi$-mixing where $\phi$ decays exponentially fast.
Therefore we can apply Theorem~\ref{main.theorem}.

A simple example of a Gibbs Markov map is the Gauss map, where $\Omega=(0,1]$, 
$Tx=\frac1x \mod 1$, the measure $\mu$ is the Gauss measure given by $d\mu(x)=\frac1{\log 2}\frac{dx}{1+x}$
 and the partition is $\mathbb{A}=\{(\frac1{j+1},\frac1j]: j\in\mathbb{N}\}$.

\end{document}